\documentclass[10pt,leqno]{amsart}
\usepackage{amssymb,amsthm}
\usepackage{amsmath}
\usepackage{setspace}
\usepackage{dcpic,pictexwd}
\usepackage[all]{xy}
\usepackage[pdftex]{graphicx}
\usepackage{mathrsfs}
\usepackage[pdftex]{hyperref}
\usepackage{bbm}

\usepackage{lscape}

\theoremstyle{definition}
 \newtheorem{definition}{Definition}[section]

\theoremstyle{plain}
 \newtheorem{proposition}[definition]{Proposition}

\theoremstyle{plain}
 \newtheorem{theorem}[definition]{Theorem}
 
\theoremstyle{definition}
 \newtheorem{example}[definition]{Example}

\theoremstyle{plain}
 \newtheorem{lemma}[definition]{Lemma}

\theoremstyle{plain}

\theoremstyle{remark}
 \newtheorem{remark}[definition]{Remark}

\theoremstyle{definition}

\theoremstyle{plain}

\newcommand{\Hom}{\mathrm{Hom}}

\newcommand{\Z}{\mathbb{Z}}

\newcommand{\A}{\Lambda}

\renewcommand{\k}{\Bbbk}
\renewcommand{\1}{\mathbbm{1}}

\hypersetup{
    bookmarks=true,         % show bookmarks bar?
    unicode=false,          % non-Latin characters in AcrobatÕs bookmarks
    pdftoolbar=true,        % show AcrobatÕs toolbar?
    pdfmenubar=true,        % show AcrobatÕs menu?
    pdffitwindow=false,     % window fit to page when opened
    pdfstartview={FitH},    % fits the width of the page to the window
    pdftitle={On Auslander-Reiten components of string and band complexes for a certain class of symmetric special biserial algebras},    % title
    pdfauthor={Jose A. Velez-Marulanda},     % author
    pdfsubject={},   % subject of the document
    pdfcreator={Jose A. Velez-Marulanda},   % creator of the document
    pdfproducer={Jose A. Velez-Marulanda}, % producer of the document
    pdfkeywords={Universal deformation rings} {Singular Equivalences of Morita Type with Level}{Finitely generated Gorenstein-projective modules}, % list of keywords
    pdfnewwindow=true,      % links in new window
    colorlinks=true,       % false: boxed links; true: colored links
    linkcolor=red,          % color of internal links
    citecolor=blue,        % color of links to bibliography
    filecolor=magenta,      % color of file links
    urlcolor=blue           % color of external links
}
\title[On Auslander-Reiten components of string complexes]{On Auslander-Reiten components of string complexes for a certain class of symmetric special biserial algebras}

\author{Hern\'an Giraldo}
\address{Instituto de Matem\'aticas, Universidad de Antioquia, Medell\'{\i}n, Colombia}
\email{hernan.giraldo@udea.edu.co}

\author{Ricardo  Rueda-Robayo}
\address{Instituto de Matem\'aticas, Universidad de Antioquia, Medell\'{\i}n, Colombia}
\email{ricardo.rueda@udea.edu.co}

\author{Jos\'e A. V\'elez-Marulanda}
\address{Department of Mathematics, Valdosta State University, Valdosta, GA, United States of America}
\email{javelezmarulanda@valdosta.edu (Corresponding author)}

\thanks{This research was partly supported by the Faculty Scholarship of the Office of Academic Affairs at the Valdosta State University, by CODI and Estrategia de Sostenibilidad 2020-2021 (Universidad de Antioquia, UdeA), and COLCIENCIAS (CONVOCATORIA DOCTORADOS NACIONALES N0. 727 de 2015).}

\begin{document}
\maketitle
\begin{abstract}
Let $\k$ be an algebraically closed field. In this article, inspired by the description of indecomposable objects in the derived category of a gentle algebra obtained by V. Bekkert and H. A. Merklen, we define string complexes for a certain class $\mathscr{C}$ of symmetric special biserial algebras, which are indecomposable perfect complexes in the corresponding derived category. We also prove that if $\A$ is a $\k$-algebra in the class $\mathscr{C}$ and $P^\bullet$ is a string complex over $\A$, then $P^\bullet$ lies in the rim of its Auslander-Reiten component. 
\end{abstract}
\renewcommand{\labelenumi}{\textup{(\roman{enumi})}}
\renewcommand{\labelenumii}{\textup{(\roman{enumi}.\Alph{enumii})}}
\numberwithin{equation}{section}

\section{Introduction}\label{sec1}

Throughout this article, we assume that $\k$ is a fixed algebraically closed field of arbitrary characteristic.  Let $\A$ be an arbitrary but fixed finite dimensional $\k$-algebra. We denote by $\textup{mod}$-$\A$ the abelian category of finitely generated right $\A$-modules, and by $\textup{proj-}\A$ the full subcategory of $\textup{mod-}\A$ whose objects are projective modules. We denote by  $\mathcal{D}^b(\textup{mod-}\A)$ the bounded derived category of $\A$ and by $\mathcal{K}^b(\textup{proj-}\A)$ the full subcategory of $\mathcal{D}^b(\textup{mod-}\A)$ of perfect complexes over $\A$. It follows from \cite[Thm. 1.3]{bekkert-drozd} that $\A$ is either derived tame (in the sense of \cite{geiss2}) or derived wild (see \cite[Def. 1.2]{bekkert-drozd}).  This raises the question of classifying all derived tame algebras up to derived equivalence, which in turn rises the question of the classification of indecomposable objects in derived categories over finite dimensional algebras. In \cite{bekmerk}, V. Bekkert and H. A. Merklen provided a complete classification of the indecomposable objects of  $\mathcal{D}^b(\textup{mod-}\A)$ for when $\A$ is a gentle algebra (as introduced in \cite{assem}) by using so-called string and band complexes.  Later in \cite{bekmerk2},  together with E. N. Marcos, they extended this description in order to classify the indecomposable objects in $\mathcal{D}^b(\textup{mod-}\A)$ for when $\A$ is a skew-gentle algebra (as introduced in \cite{geiss1}).  An important consequence of this approach is that gentle and skew-gentle algebras are derived tame (see \cite[Thm. 4]{bekmerk} and \cite[Cor. 5]{bekmerk2}). This fact was used recently by V. Bekkert together with the first and third authors in \cite[Thm .1.3]{bekkert-giraldo-velez2} to prove that a cycle Nakayama algebra is derived tame if and only if it is either gentle or derived equivalent to a skew-gentle algebra.  Although derived tameness of algebras has been studied by many authors, examples of algebras that are derived tame are scarce in the literature (see e.g. the introduction of \cite{bekkert-drozd-futorny} and the references within).  Recall that $\A$ is said to be self-injective if the regular right $\A$-module $\A_\A$ is injective, and that $\A$ is called a Frobenius algebra provided that the left $\A$-modules ${_\A}\A$ and $(\A_\A)^\ast=\mathrm{Hom}_\k(\A_\A,\k)$ are isomorphic. Recall also that $\A$ is said to be a symmetric algebra provided that $\A$ is Frobenius and that there exists a non-degenerate associative bilinear form $\theta: \A\times \A\to \k$ with $\theta(a,b)=\theta(b,a)$ for all $a,b\in \A$. By \cite
[Prop. 9.9]{curtis}, every Frobenius $\k$-algebra is self-injective, which implies that every symmetric $\k$-algebra is also self-injective. In \cite[Cor. 2.5]{bautista1}, R. Bautista proved that if $\A$ is a self-injective algebra, then $\A$ is either derived discrete (in the sense of \cite{vossieck}) or derived wild. Note that by \cite[Def. 1.2]{bekkert-drozd}, every derived discrete finite dimensional $\k$-algebra is derived tame. On the other hand, it follows by either \cite[Lemma 3.2]{bekkert-giraldo-velez2} or \cite[Prop. 4.1]{czhang} that if $\A$ is a  self-injective Nakayama algebra, then $\A$ is derived tame if and only if it is gentle, and thus in this situation $\A$ is actually derived discrete. Following the classification up to derived equivalence provided by G. Bobi\'nski et al. in \cite{bobin2}, it follows that most self-injective algebras (up to derived equivalence) are derived wild.  Thus, the non-trivial description of indecomposable objects in $\mathcal{D}^b(\textup{mod-}\A)$ for when $\A$ is a non-gentle self-injective algebra is rather a challenging task. On the other hand, in \cite{giraldo-velez}, by using the ideas of Bekkert and Merklen in \cite{bekmerk}, the first and the third author defined string and band complexes for the $\k$-algebra $\A_4$ as in Figure \ref{fig1}, and proved that these string and band complexes are also indecomposable objects in $\mathcal{K}^b(\textup{proj-}\A_4)$. Moreover, by using the results obtained by P. Webb in \cite{webb} that concern complexes over self-injective algebras, they provided a full description of the components of the Auslander-Reiten quiver of $\mathcal{K}^b(\textup{proj-}\A_4)$ that contain either a string or a band complex over $\A_4$. This $\k$-algebra $\A_4$ is of dihedral type (as introduced by K. Erdmann in \cite{erdmann}), and thus it is in particular a symmetric special biserial algebra in the sense of \cite{wald}.

%More recently, the first author together with A. Franco and P. Rizzo  introduced the class of {\it string almost gentle algebras}, and proved that for when $\A=\k Q/I$ is either string almost gentle or a string algebra such that each arrow of $Q$ belongs to a unique maximal path in $\A$, then string and band complexes are definable for $\A$ and they are also indecomposable in $\mathcal{K}^b(\textup{proj-}\A)$.  As in \cite{giraldo-velez}, in order to prove this, they followed the ideas in \cite{bekmerk} and defined a functor $\mathbf{F}_\A: \mathcal{K}^b(\textup{proj-}\A)\to\mathscr{S}(\mathscr{Y}(\A), \k)$, where $\mathscr{S}(\mathscr{Y}(\A), \k)$ is the $\k$-category of Bondarenko's representations of a linearly ordered set $\mathscr{Y}(\A)$ determined by $\A$ (see e.g. \cite[\S 2 \& \S4.2]{bekmerk}), which identifies string and band complexes for $\A$ with indecomposable representations in $\mathscr{S}(\mathscr{Y}(\A), \k)$.  

In this article, we define string complexes for symmetric special biserial algebras $\A$ that satisfy the following condition:
\begin{itemize}
\item[(C)] The $\k$-algebra  $\A$ is of the form $\k Q/I$, where the admissible ideal  $I$ of the path algebra $\k Q$ has  a minimal set of generators given by 
\begin{equation*}
\rho =\{\alpha\beta, p_1-p_2\,|\,\alpha, \beta \in Q_1, p_1,p_2\in \mathbf{Pa}_{>1}(\k Q) \text{ with } \mathbf{s}(p_1)=\mathbf{s}(p_2), \mathbf{t}(p_1)=\mathbf{t}(p_2)\}.
\end{equation*} 
\end{itemize}

Note that the symmetric special biserial algebras in Figure \ref{fig1} satisfy the condition (C). Moreover, if $\A$ is a symmetric special biserial algebra, then it follows from \cite{roggen} and \cite{schroll} that $\A$ is also a Brauer graph algebra (see e.g. \cite[\S 2]{schroll2} for the definition), and thus by the discussion in e.g. \cite[\S 2.4]{schroll}, many symmetric special biserial algebras $\A$ satisfy the condition (C). 

On the other hand, it follows by Remark \ref{stringalg} below, that if $\A$ is a special biserial $\k$-algebra, then we can associate to $\A$ a string $\k$-algebra $\widetilde{\A}$.  Thus the non-projective indecomposable $\A$-modules can be described combinatorially by using so-called string and bands for $\widetilde{\A}$; the corresponding modules are called string and band $\A$-modules. We refer the reader to \cite{buri} (see also \cite[Chap. II]{erdmann}) for getting more information regarding the description and the properties of these string and band modules, and  to \cite{krause} for a description of the morphisms between these objects.  

\begin{definition}\label{classC}
We denote by $\mathscr{C}$ the class of symmetric special biserial algebras $\A=\k Q/I$ that satisfy the condition (C) together with the property that every arrow in $Q$ belongs to a unique maximal path in the associated string algebra $\widetilde{\A}$ corresponding to $\A$. 
\end{definition}

%We say that a special biserial $\k$-algebra $\A$ that satisfy the condition (C) is of {\it Boundarenko's type} if there exists such a functor $\mathbf{F}_\A$ as above (see Definition \ref{algBoundarenko}). It follows that gentle algebras and string almost gentle algebras are of Boundarenko's type. 

\begin{remark}
It is easy to check that all the $\k$-algebras in Figure \ref{fig1} with the exception of $\A_4$ belong to the class $\mathscr{C}$ as in Definition \ref{classC}. 
\end{remark}

Our main result (see Theorem \ref{prop4.2}) gives a version of \cite[Thm. 14]{giraldo-velez} for all symmetric special biserial $\k$-algebras $\A$ that belong to the class $\mathscr{C}$ as in Definition \ref{classC}. More precisely, we prove that if $P^\bullet$ is a string complex over such $\k$-algebra $\A$, then $P^\bullet$ is indecomposable in $\mathcal{K}^b(\textup{proj-}\A)$, and if $\mathfrak{C}$ is the component of the Auslander-Reiten quiver of $\mathcal{K}^b(\textup{proj-}\A)$ containing $P^\bullet$, then $P^\bullet$ lies in the rim of $\mathfrak{C}$.  We next use this to describe the representatives of the orbits of the Auslander-Reiten translation in $\mathfrak{C}$ .  It is important to mention that by \cite[Thm. 3.7]{wheeler} (see also \cite[Thm. 5.4]{happkellrei}), $\mathfrak{C}$ is of the form $\Z\mathbb{A}_\infty$. 

\begin{figure}[htb]
\begin{align*}
Q^{(1)} &=\hspace*{1.5cm} \xymatrix@1@=18pt{
	\underset{0}{\bullet}\ar@<1ex>[r]^{\tau_0}&\underset{1}{\bullet}\ar@<1ex>[l]^{\gamma_1}\ar@<1ex>[r]^{\tau_1}& \underset{2}{\bullet}\ar@<1ex>[l]^{\gamma_2}
	}&
Q^{(2)}(m)&=\hspace*{0.5cm}\xymatrix@1@=18pt{
&\underset{0}{\bullet}\ar@/^/[dl]^{\bar{a}_{m-1}}\ar@/^/[rr]^{a_0}&&\underset{1}{\bullet}\ar@/^/[ll]^{\bar{a}_0}\ar@/^/[dr]^{a_1}&\\
\underset{m-1}{\bullet}\ar@/^/[ur]^{a_{m-1}}\ar@/^/[d]^{\bar{a}_{m-2}}&&&	&\underset{2}{\bullet}\ar@/^/[ul]^{\bar{a}_1}\ar@/^/[d]^{a_2}\\
\underset{m-2}{\bullet}\ar@/^/[u]^{a_{m-2}}\ar@/^/@{.>}[dr]^{}&&&	&\underset{3}{\bullet}\ar@/^/[u]^{\bar{a}_2}\ar@/^/@{.>}[dl]^{}\\
&\underset{\ast}{\bullet}\ar@/^/@{.>}[ul]^{}\ar@/^/@{.>}[rr]^{}&&\underset{\ast}{\bullet}\ar@/^/@{.>}[ll]^{}\ar@/^/@{.>}[ur]^{}&
}\\\\
Q^{(3)} &=\xymatrix@1@=20pt{
	\ar@(ul,dl)_{\zeta_0} \underset{0}{\bullet}\ar@<1ex>[r]^{\tau_0}&\underset{1}{\bullet}\ar@<1ex>[l]^{\gamma_1}\ar@<1ex>[r]^{\tau_1}& \underset{2}{\bullet}
\ar@<1ex>[l]^{\gamma_2}\ar@(ur,dr)^{\zeta_2} 
	}
&
Q^{(4)} &=\hspace*{0.5cm} \xymatrix@1@=18pt{
	\ar@(ul,dl)_{\zeta_0}\underset{0}{\bullet}\ar[rr]^{\tau_0}&&\underset{1}{\bullet}\ar[dl]^{\tau_1}\ar@(ur,dr)^{\zeta_1}\\
	&\underset{2}{\bullet}\ar[ul]^{\tau_2}\ar@(dl,dr)_{\zeta_2}&\\
	}
\end{align*}
\caption{Quivers of special biserial algebras satisfying condition (C).}\label{fig0}
\end{figure}
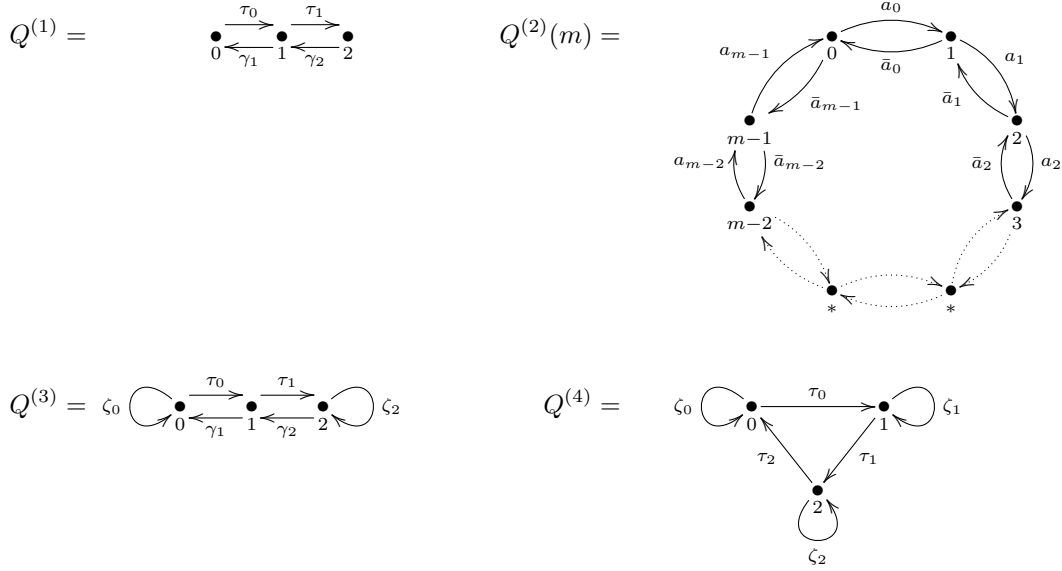

\begin{figure}[htb]
\begin{align*}
\A_1 &= \k Q^{(1)}/\langle \tau_0\tau_1, \gamma_2\gamma_1, \gamma_1\tau_0-\tau_1\gamma_2\rangle\\
\A_{2,m}&= \k Q^{(2)}(m)/ \langle a_ia_{i+1}, \bar{a}_i\bar{a}_{i-1}, a_i\bar{a}_i-\bar{a}_{i-1}a_{i-1} :i\in \Z/m\rangle, \text{and $m\geq 1$}\\
\A_3&= \k Q^{(3)}/\langle \gamma_1\zeta_0,\zeta_0\tau_0,\tau_0\tau_1,\gamma_2\gamma_1,\tau_1\zeta_2,\zeta_2\gamma_2,\tau_0\gamma_1-\zeta_0^2,
\gamma_2\tau_1-\zeta_2^2,\gamma_1\tau_0-\tau_1\gamma_2\rangle\\
\A_4&= \k Q^{(4)}/\langle \zeta_0\tau_0,\tau_0\zeta_1,\zeta_1\tau_1,\tau_1\zeta_2,\zeta_2\tau_2,\tau_2\zeta_0,\tau_0\tau_1\tau_2-\zeta_0^2, 
\tau_1\tau_2\tau_0-\zeta_1^2, \tau_2\tau_0\tau_1-\zeta_2^2\rangle
\end{align*}
\caption{Examples of special biserial algebras satisfying condition (C).}\label{fig1}
\end{figure}
 
This article is organized as follows. In \S \ref{sec2}, we recall the definitions of special biserial and string algebras, the definition and some properties of Auslander-Reiten triangles, and summarize the properties of  the components in the Auslander-Reiten quiver of the category of perfect complexes over a symmetric algebra from P. Webb's preprint \cite{webb} and from W.W. Wheeler's article \cite{wheeler}. In \S \ref{sec3}, we define string complexes for the symmetric special biserial algebras $\A$ that satisfy the condition (C). Finally, in \S \ref{sec4}, we prove Theorem \ref{prop4.2}.

This article also constitutes the doctoral dissertation of the second author under the supervision of the other two.

\section{Preliminares}\label{sec2}
Throughout this section we keep the notation introduced in \S \ref{sec1}, i.e. $\k$ is an algebraically closed field and $\A$ is a a finite dimensional $\k$-algebra. If $f:X\to Y$ and $g:Y\to Z$ are morphisms in a category $\mathcal{C}$, we denote by $fg$ the composition of $f$ with $g$.

\subsection{Quiver with relations, and path algebras}
Recall that a quiver $Q$ is a directed graph with a set of vertices $Q_0$, a set of arrows $Q_1$ and two functions $\mathbf{s},\mathbf{t}:Q_1\to Q_0$, where for all $\alpha\in Q_1$, $\mathbf{s}\alpha$ (resp. $\mathbf{t}\alpha$) denotes the vertex where $\alpha$ starts (resp. ends). A path in $Q$ of length $n\geq 1$ is an ordered sequence of arrows $w=\alpha_1\cdots\alpha_n$ with $\mathbf{t}\alpha_j=\mathbf{s}\alpha_{j+1}$ for $1\leq j <n$. In particular, we write paths from left to right. Additionally, for each $v\in Q_0$, we have a trivial path $\1_v$ of length zero with $\mathbf{s}\1_v=v=\mathbf{t}\1_v$. For a non-trivial path $w=\alpha_1\cdots \alpha_n$ we define $\mathbf{s}w=\mathbf{s}\alpha_1$ and $\mathbf{t}w=\mathbf{t}\alpha_n$.   A non-trivial path $w$ in $Q$ is said to be an oriented cycle provided that $\mathbf{s}w=\mathbf{t}w$. The path algebra $\k Q$ of a quiver $Q$ is the $\k$-vector space whose basis consists of all the paths in $Q$, and for two paths $w$ and $w'$, their multiplication is given by the concatenation $ww'$ provided that $\mathbf{t}w=\mathbf{s}w'$, or zero otherwise. Let $J$ be the two-sided ideal of $\k Q$ generated by all the arrows in $Q$. We say that an ideal $I$ of $\k Q$ is admissible if there exists $d\geq 2$ such that $J^d\subseteq I \subseteq J^2$. In this situation, the quotient $\k Q/ I$ is a finite dimensional $\k$-algebra. If $w$ is a path in $Q$, we denote also by $w$ its equivalence class in $\k Q/I$. In particular, a path $w$ in $\k Q/I$ is a {\it zero-path} if and only if $w$ belongs to $I$. We say that a non-zero path $w$ in $\k Q/ I$ is {\it maximal} if for all arrows $\alpha, \beta\in Q_1$ such that $\mathbf{t}\alpha= \mathbf{s}w$ and  $\mathbf{s}\beta=\mathbf{t}w$, we have that $\alpha w \beta $ is a zero path in $\k Q/ I$. 

From now on we assume that $\A = \k Q/ I$, where $Q$ is a finite quiver and $I$ is an admissible ideal of $\k Q$. 
For each $v\in Q_0$, we denote also by $\1_v$ the corresponding primitive idempotent in $\A$, by $S_v$ the corresponding simple right $\A$-module and by $\mathbf{P}_v$ the corresponding indecomposable projective right $\A$-module, i.e., $\mathbf{P}_v = \1_v\A$. We denote by $\mathbf{Pa}(\A)$ the sets of all non-zero paths, and for all integers $n\geq 0$, we denote  by $\mathbf{Pa}_{>n}(\A)$ the set of all paths whose length is greater than $n$. 
\begin{remark}\label{projmorph}
Let $w \in \mathbf{Pa}(\A)$. Then $w$ induces a morphism $p(w)$ of right $\A$-modules from $\mathbf{P}_{\mathbf{s}w}$ to $\mathbf{P}_{\mathbf{t}w}$ defined as $p(w)(u) = wu$ for all $u\in \mathbf{P}_{\mathbf{s}w}$. Moreover, the $\k$-vector space $\Hom_\A(\mathbf{P}_{\mathbf{s}w}, \mathbf{P}_{\mathbf{t}w})$ is generated by morphisms of this kind. 
\end{remark}

\subsection{Biserial, special biserial and string algebras}\label{gentle}

Following \cite{skow2}, $\A$ is a {\it biserial}  $\k$-algebra provided that the radical of any indecomposable non-uniserial projective, left or right, $\A$-module is a sum of two uniserial submodules whose intersection is simple or zero. On the other hand, following \cite{wald} (see also \cite{skow2}), $\A$ is {\it special biserial} if the following conditions are satisfied:
\begin{itemize}
\item[(SB1)] For any vertex $v\in Q$, there are at most two arrows ending at $v$, and at most two arrows starting at $v$.
\item[(SB2)] Given an arrow $\alpha \in Q_1$, there is a most an arrow $\beta$ with $\mathbf{s}\beta = \mathbf{t}\alpha$ such that  $\alpha\beta\not\in I$, and there is a most an arrow $\gamma$ with $\mathbf{s}\alpha = \mathbf{t}\gamma$ such that  $\gamma\alpha\not\in I$.
\end{itemize}

By \cite[Lemma 1]{skow2}, any special biserial algebra is biserial. However, if $\A=\k Q/I$ where 
\begin{align*}
Q=\xymatrix@1@=20pt{
\underset{1}{\bullet}\ar[rr]^{\alpha}&&\underset{2}{\bullet}\ar[rr]^{\beta}\ar[dr]_{\gamma}&&\underset{1}{\bullet}\ar[rr]^{\epsilon}&&\underset{5}{\bullet}\\
&&&\underset{3}{\bullet}\ar[ur]_{\delta}&&&
}
&&\text{and}&&
I=\langle\beta\epsilon, \alpha\beta-\alpha\gamma\delta\rangle,
\end{align*} 
\noindent
then by the arguments in \cite[pg. 175]{skow2}, $\A$ provides an example of a biserial $\k$-algebra that is not special biserial. 

We say that $\A$ is a {\it string algebra} provided that $I$ is generated by only zero-relations and in addition satisfies the above conditions (SB1) and (SB2). 

\begin{remark}\label{stringalg}
By e.g. \cite[\S II.1.3]{erdmann}, in order to study indecomposable non-projective right $\A$-modules and irreducible morphisms over a special biserial algebra $\A$, we can always do this by looking at the quotient algebra $\widetilde{\A}=\A/S_0$, where $S_0 = \bigoplus_{v\in L}\mathrm{soc}\, \A\1_v$  and $L=\{v\in Q_0: \A\1_v \text{ is injective and not uniserial}\}$. In this situation, it follows that $\widetilde{\A}$ is a string algebra and we call it the {\it associated string algebra} of $\A$. In particular, if all the indecomposable injective left $\A$-modules are all non-uniserial, then $S_0$ is the socle of $\A$, and thus  $\widetilde{\A}=\A/\mathrm{soc}\,\A$.  
\end{remark}

\begin{example}\label{exam1}
Let consider the special biserial algebras as in Figure \ref{fig1}. Then 
\begin{align*}
\widetilde{\A}_1 &= \k Q^{(1)}/\langle \tau_0\tau_1, \gamma_2\gamma_1, \gamma_1\tau_0,\tau_1\gamma_2\rangle\\
\widetilde{\A}_{2,m}&= \k Q^{(2)}(m)/ \langle a_ia_{i+1}, \bar{a}_i\bar{a}_{i-1}, a_i\bar{a}_i,\bar{a}_{i-1}a_{i-1} :i\in \Z/m\rangle, \text{and $m\geq 1$}\\
\widetilde{\A}_3&= \k Q^{(3)}/\langle \gamma_1\zeta_0,\zeta_0\tau_0,\tau_0\tau_1,\gamma_2\gamma_1,\tau_1\zeta_2,\zeta_2\gamma_2,\tau_0\gamma_1,\zeta_0^2,
\gamma_2\tau_1,\zeta_2^2,\gamma_1\tau_0,\tau_1\gamma_2\rangle\\
\widetilde{\A}_4&= \k Q^{(4)}/\langle \zeta_0\tau_0,\tau_0\zeta_1,\zeta_1\tau_1,\tau_1\zeta_2,\zeta_2\tau_2,\tau_2\zeta_0,\tau_0\tau_1\tau_2,\zeta_0^2, 
\tau_1\tau_2\tau_0,\zeta_1^2, \tau_2\tau_0\tau_1,\zeta_2^2\rangle
\end{align*}
\end{example}

%The corresponding indecomposable $\A_{m,N}$-modules are called string and band modules. In the following, we describe these string $\A_{m,N}$-modules, the components of the stable Auslander-Reiten quiver $\Gamma_s(\A_{m,N})$ containing them as determined in \cite{erdmann3}, and the morphisms between them as determined in \cite{buri,krause}.   

%\subsection{String modules for $\A_{m,N}$}\label{ape1}

\subsection{Auslander-Reiten components containing perfect complexes over symmetric algebras}

As stated before, we denote by $\mathcal{D}^b(\textup{mod-}\A)$ the bounded derived category of $\A$. We denote by  $\mathcal{K}^{-,b}(\textup{proj-}\A)$ the category of bounded above complexes whose terms are in proj-$\A$, with at most finitely many non-zero cohomology groups, and by $\mathcal{K}^b(\textup{proj-}\A)$ the homotopy category of perfect complexes over $\A$.  We denote by $T$ the shifting functor on $\mathcal{D}^b(\textup{mod-}\A)$ (resp. $\mathcal{K}^b(\textup{proj-}\A)$, resp. $\mathcal{K}^b(\textup{proj-}\A)$) i.e., $T$ shifts complexes one place to the left and changes the sign of the differential (see e.g. \cite[Chap. I]{hartshorne}). It is well-known that $\mathcal{D}^b(\textup{mod-}\A)$ (resp. $\mathcal{K}^b(\textup{proj-}\A)$, resp. $\mathcal{K}^{-,b}(\textup{proj-}\A)$) is a triangulated category in the sense of \cite{verdier}, and that $\mathcal{D}^b(\textup{mod-}\A)$ is equivalent to $\mathcal{K}^{-,b}(\textup{proj-}\A)$ as triangulated categories. Following \cite[Chap. I, \S 4]{happel}, a distinguished triangle $X^\bullet\xrightarrow{u^\bullet} Y^\bullet\xrightarrow{v^\bullet}Z\xrightarrow{w^\bullet} T(X^\bullet)$ in $\mathcal{D}^b(\textup{mod}\,\A)$ is called an {\it Auslander-Reiten triangle} if the following 
conditions are satisfied.
\begin{enumerate}
\item The objects $X^\bullet$ and $Z^\bullet$ are indecomposable.
\item The morphism $w^\bullet$ is non-zero.
\item If $f^\bullet: W^\bullet\to Z^\bullet$ is not a retraction, then there exists $f'^\bullet:W\to Y^\bullet$ such that $v^\bullet f'^\bullet=f^\bullet$.
\end{enumerate}

\begin{remark} 
From now on we assume that $\A$ is a symmetric $\k$-algebra. 
\end{remark}

It follows by e.g \cite[Prop. 3.8(b)]{auslander} that the Nakayama functor $\nu_\A= D\mathrm{Hom}_{\A}(-,\A)$, where $D=\Hom_\k(-,\k)$,  is naturally equivalent to the identity functor. Moreover, it also follows from the results in \cite{happel2} that for all indecomposable 
objects $Z^\bullet$ in $\mathcal{K}^{-,b}(\textup{proj-}\A)$, there exists an Auslander-Reiten triangle ending in $Z^\bullet$ if and only if $Z^\bullet$ is an object of $\mathcal{K}^b(\textup{proj-}\A)$, and  
this triangle is of the form  $T^{-1}(Z^\bullet)\to Y^\bullet\to Z^\bullet\to Z^\bullet$. Therefore, we can assume that every Auslander-Reiten triangle in $\mathcal{K}^b(\textup{proj-}\A)$ is isomorphic 
to 
\begin{equation*}
T^{-1}(Z^\bullet)\to T^{-1}(\mathrm{cone}(h^\bullet))\to Z^\bullet\xrightarrow{h^\bullet} Z^\bullet,
\end{equation*}
for some object $Z^\bullet$ in $\mathcal{K}^b(\textup{proj-}\A)$ and some morphism $h^\bullet:Z^\bullet\to Z^\bullet$, where $\mathrm{cone}(h^\bullet)$ denotes the mapping cone of $h^\bullet$. 
We denote by $\Gamma(\mathcal{K}^b(\textup{proj-}\A))$ the Auslander-Reiten quiver of  $\mathcal{K}^b(\textup{proj-}\A)$. We say that a complex $Z^\bullet$ lies on the {\it rim} of its component in $\Gamma(\mathcal{K}^b(\textup{proj-}\A))$, if in the Auslander-Reiten triangle $X^\bullet\to Y^\bullet\to Z^\bullet\to TX^\bullet$, the complex $Y^\bullet$ is indecomposable. As stated before, it follows from \cite[Thm. 3.7]{wheeler} (see also \cite[Thm. 5.4]{happkellrei}) that if $\mathfrak{C}$ is a 
connected component of $\Gamma(\mathcal{K}^b(\textup{proj-}\A))$, then $\mathfrak{C}$ is of the form $\mathbb{Z}\mathbb{A}_\infty$. Thus, the component $\mathfrak{C}$ of $
\Gamma(\mathcal{K}^b(\textup{proj-}\A))$ with a complex $C_0^\bullet$ lying on its rim looks as in Figure \ref{fig5}.

\begin{figure}
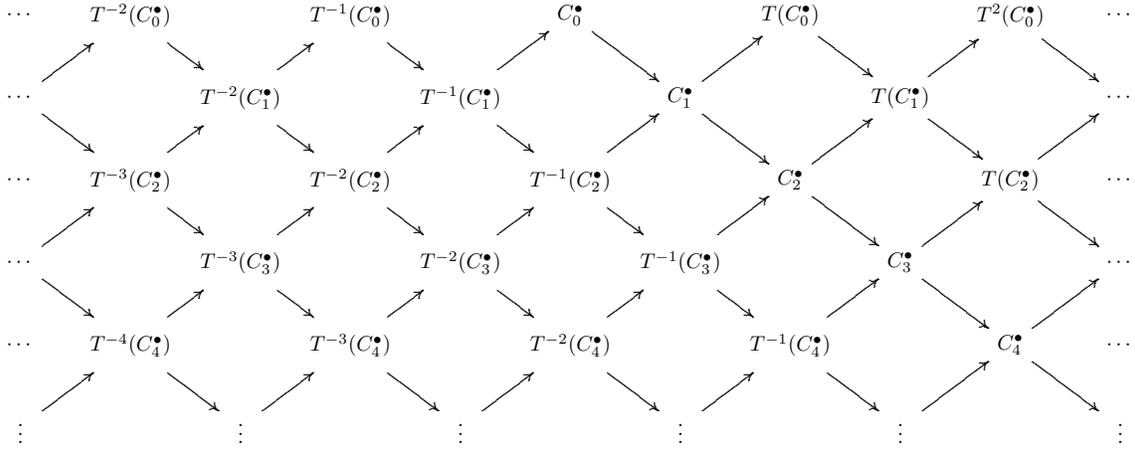

\scalebox{0.8}{
$
\begindc{\commdiag}[130]
% Second row
\obj(-20,7)[2]{$\cdots$}
\obj(-16,7)[3]{$T^{-2}(C_0^\bullet)$}
\obj(-8,7)[4]{$T^{-1}(C_0^\bullet)$}
\obj(0,7)[5]{$C_0^\bullet$}
\obj(8,7)[6]{$T(C_0^\bullet)$}
\obj(16,7)[7a]{$T^2(C_0^\bullet)$}
\obj(20,7)[7b]{$\cdots$}
%End Second Row

%\mor{7c}{3}{}
%Third Row
\obj(-20,4)[7c]{$\cdots$}
\obj(-12,4)[8]{$T^{-2}(C_1^\bullet)$}
\obj(-4,4)[9]{$T^{-1}(C_1^\bullet)$}
\obj(4,4)[10]{$C_1^\bullet$}
\obj(12,4)[11]{$T(C_1^\bullet)$}
\obj(20,4)[11a]{$\cdots$}
%End Third Row
\mor{11}{7a}{}
\mor{7a}{11a}{}
\mor{3}{8}{}
\mor{7c}{3}{}
% Fourth Row
\obj(-20,1)[12a]{$\cdots$}
\obj(-16,1)[12]{$T^{-3}(C_2^\bullet)$}
\obj(-8,1)[13]{$T^{-2}(C_2^\bullet)$}
\obj(0,1)[14]{$T^{-1}(C_2^\bullet)$}
\obj(8,1)[15]{$C_2^\bullet$}
\obj(16,1)[16]{$T(C_2^\bullet)$}
\obj(20,1)[16a]{$\cdots$}

%End Fourth Row
\mor{7c}{12}{}
\mor{12}{8}{}
\mor{11}{16}{}
\mor{16}{11a}{}
%Fifth Row
\obj(-20,-2)[17a]{$\cdots$}
\obj(-12,-2)[17]{$T^{-3}(C_3^\bullet)$}
\obj(-4,-2)[18]{$T^{-2}(C_3^\bullet)$}
\obj(4,-2)[19]{$T^{-1}(C_3^\bullet)$}
\obj(12,-2)[20]{$C_3^\bullet$}
\obj(20,-2)[20a]{$\cdots$}
\mor{17a}{12}{}
\mor{12}{17}{}
\mor{20}{16}{}
\mor{16}{20a}{}

%End Fifth Row

%Sixth Row
\obj(-20,-5)[21a]{$\cdots$}
\obj(-16,-5)[21]{$T^{-4}(C_4^\bullet)$}
\obj(-8,-5)[22]{$T^{-3}(C_4^\bullet)$}
\obj(0,-5)[23]{$T^{-2}(C_4^\bullet)$}
\obj(8,-5)[24]{$T^{-1}(C_4^\bullet)$}
\obj(16,-5)[25]{$C_4^\bullet$}
\obj(20,-5)[25a]{$\cdots$}
%End Sixth Row
\mor{17a}{21}{}
\mor{21}{17}{}
\mor{20}{25}{}
\mor{25}{20a}{}
%Seventh Row
\obj(-20,-8)[26a]{$\vdots$}
\obj(-12,-8)[26]{$\vdots$}
\obj(-4,-8)[27]{$\vdots$}
\obj(4,-8)[28]{$\vdots$}
\obj(12,-8)[29]{$\vdots$}
\obj(20,-8)[29a]{$\vdots$}
\mor{26a}{21}{}
\mor{21}{26}{}
\mor{29}{25}{}
\mor{25}{29a}{}

%End Seventh Row

%Morphisms

\mor{8}{4}{}
\mor{4}{9}{}
\mor{9}{5}{}
\mor{5}{10}{}
\mor{10}{6}{}
\mor{6}{11}{}
\mor{8}{13}{}
\mor{13}{9}{}
\mor{9}{14}{}
\mor{14}{10}{}
\mor{10}{15}{}
\mor{15}{11}{}
\mor{17}{13}{}
\mor{13}{18}{}
\mor{18}{14}{}
\mor{14}{19}{}
\mor{19}{15}{}
\mor{15}{20}{}
\mor{17}{22}{}
\mor{22}{18}{}
\mor{18}{23}{}
\mor{23}{19}{}
\mor{19}{24}{}
\mor{24}{20}{}
\mor{26}{22}{}
\mor{22}{27}{}
\mor{27}{23}{}
\mor{23}{28}{}
\mor{28}{24}{}
\mor{24}{29}{}
\enddc
$
}
\caption{Component $\mathfrak{C}$ of $\Gamma(\mathcal{K}^b(\textup{proj-}\A))$ near the complex $C_0^\bullet$.}\label{fig5}
\end{figure}
\noindent

\begin{remark}\label{remmin}
If $P^\bullet$ is an object in $\mathcal{K}^b(\textup{proj-}\A)$, then there exists $m\geq 0$ such that 
\begin{equation}\label{complex}
P^\bullet = \cdots\to 0\to P^n\xrightarrow{\delta_{P}^n}P^{n+1}\to\cdots \to P^{n+m-1}\xrightarrow{\delta_{P}^{n+m-1}}P^{n+m}\to 0\to \cdots, 
\end{equation}
and $\delta_{P}^i\delta_{P}^{i+1}=0$ for all $i\in \Z$. 
Without loss of generality, we can assume that $P^\bullet$ is {\it minimal}, i.e., for all $n\leq i\leq n+m-1$, $\mathrm{im}(\delta^i_{P})\subseteq \mathrm{rad}(P^{i+1})$, for every complex in $\mathcal{K}^b(\textup{proj-}\A)$ can be written as the sum of one complex having this property and another one whose differential maps are either zero or isomorphisms (see e.g. \cite[Thm. 5]{giraldo-merklen}).  
\end{remark}
\begin{remark}\label{rem2.2}
Since $\A$ is a symmetric algebra, it follows that for all projective right $\A$-modules $P$, there is an isomorphism between the right $\A$-modules $\mathrm{top}\,P=P/\mathrm{rad}\,P$ and $\mathrm{soc}\, P$.  In this situation, we denote by $f_P:P\to P$ the map that sends isomorphically the top of each indecomposable summand of $P$ to its corresponding socle. 
\end{remark}

The following result follows from \cite[Lemma 2.4]{wheeler} and from the fact that $\A$ is a symmetric $\k$-algebra.
\begin{lemma}\label{lem4.1}
Let $P^\bullet$ be a non-zero indecomposable object in $\mathcal{K}^b(\textup{proj-}\A)$, and let $\ell$ the largest index such that $P^\ell\not=0$. If a triangle
\begin{equation*}
T^{-1}(P^\bullet)\to Q^\bullet\to P^\bullet \xrightarrow{h^\bullet} P^\bullet
\end{equation*}
is an Auslander-Reiten triangle, then $h^\bullet$ is homotopic to a morphism $f^\bullet:P^\bullet \to P^\bullet$ such that $f^j= 0$ for $j\not=\ell$ and $f^\ell=f_{P^\ell}$, where $f_{P^\ell}$ is as in Remark \ref{rem2.2}.
\end{lemma}

\subsection{Webb's cohomology diagrams for Auslander-Reiten components}
Consider the description of the component $\mathfrak{C}$  of the Auslander-Reiten quiver containing $C_0^\bullet$ as in Figure \ref{fig5}. Following \cite{webb}, after taking cohomology groups, we obtain the induced diagram of right $\A$-modules shown in Figure \ref{fig6}, which Webb calls the {\it cohomology diagram} of $\mathfrak{C}$. 
  
\begin{figure}
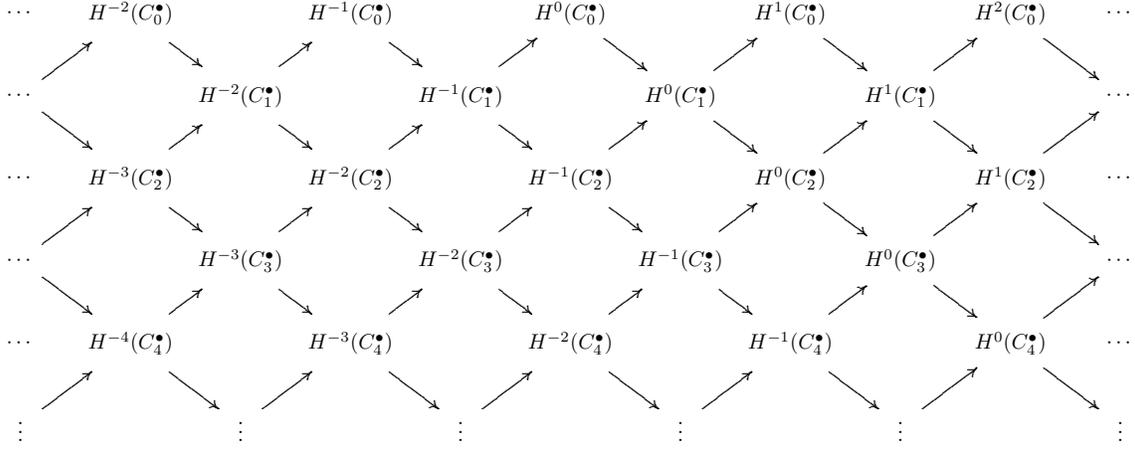

\scalebox{0.8}{
$
\begindc{\commdiag}[130]
% Second row
\obj(-20,7)[2]{$\cdots$}
\obj(-16,7)[3]{$H^{-2}(C_0^\bullet)$}
\obj(-8,7)[4]{$H^{-1}(C_0^\bullet)$}
\obj(0,7)[5]{$H^0(C_0^\bullet)$}
\obj(8,7)[6]{$H^1(C_0^\bullet)$}
\obj(16,7)[7a]{$H^2(C_0^\bullet)$}
\obj(20,7)[7b]{$\cdots$}
%End Second Row

%\mor{7c}{3}{}
%Third Row
\obj(-20,4)[7c]{$\cdots$}
\obj(-12,4)[8]{$H^{-2}(C_1^\bullet)$}
\obj(-4,4)[9]{$H^{-1}(C_1^\bullet)$}
\obj(4,4)[10]{$H^0(C_1^\bullet)$}
\obj(12,4)[11]{$H^1(C_1^\bullet)$}
\obj(20,4)[11a]{$\cdots$}
%End Third Row
\mor{11}{7a}{}
\mor{7a}{11a}{}
\mor{3}{8}{}
\mor{7c}{3}{}
% Fourth Row
\obj(-20,1)[12a]{$\cdots$}
\obj(-16,1)[12]{$H^{-3}(C_2^\bullet)$}
\obj(-8,1)[13]{$H^{-2}(C_2^\bullet)$}
\obj(0,1)[14]{$H^{-1}(C_2^\bullet)$}
\obj(8,1)[15]{$H^0(C_2^\bullet)$}
\obj(16,1)[16]{$H^1(C_2^\bullet)$}
\obj(20,1)[16a]{$\cdots$}

%End Fourth Row
\mor{7c}{12}{}
\mor{12}{8}{}
\mor{11}{16}{}
\mor{16}{11a}{}
%Fifth Row
\obj(-20,-2)[17a]{$\cdots$}
\obj(-12,-2)[17]{$H^{-3}(C_3^\bullet)$}
\obj(-4,-2)[18]{$H^{-2}(C_3^\bullet)$}
\obj(4,-2)[19]{$H^{-1}(C_3^\bullet)$}
\obj(12,-2)[20]{$H^{0}(C_3^\bullet)$}
\obj(20,-2)[20a]{$\cdots$}
\mor{17a}{12}{}
\mor{12}{17}{}
\mor{20}{16}{}
\mor{16}{20a}{}

%End Fifth Row

%Sixth Row
\obj(-20,-5)[21a]{$\cdots$}
\obj(-16,-5)[21]{$H^{-4}(C_4^\bullet)$}
\obj(-8,-5)[22]{$H^{-3}(C_4^\bullet)$}
\obj(0,-5)[23]{$H^{-2}(C_4^\bullet)$}
\obj(8,-5)[24]{$H^{-1}(C_4^\bullet)$}
\obj(16,-5)[25]{$H^0(C_4^\bullet)$}
\obj(20,-5)[25a]{$\cdots$}
%End Sixth Row
\mor{17a}{21}{}
\mor{21}{17}{}
\mor{20}{25}{}
\mor{25}{20a}{}
%Seventh Row
\obj(-20,-8)[26a]{$\vdots$}
\obj(-12,-8)[26]{$\vdots$}
\obj(-4,-8)[27]{$\vdots$}
\obj(4,-8)[28]{$\vdots$}
\obj(12,-8)[29]{$\vdots$}
\obj(20,-8)[29a]{$\vdots$}
\mor{26a}{21}{}
\mor{21}{26}{}
\mor{29}{25}{}
\mor{25}{29a}{}

%End Seventh Row

%Morphisms

\mor{8}{4}{}
\mor{4}{9}{}
\mor{9}{5}{}
\mor{5}{10}{}
\mor{10}{6}{}
\mor{6}{11}{}
\mor{8}{13}{}
\mor{13}{9}{}
\mor{9}{14}{}
\mor{14}{10}{}
\mor{10}{15}{}
\mor{15}{11}{}
\mor{17}{13}{}
\mor{13}{18}{}
\mor{18}{14}{}
\mor{14}{19}{}
\mor{19}{15}{}
\mor{15}{20}{}
\mor{17}{22}{}
\mor{22}{18}{}
\mor{18}{23}{}
\mor{23}{19}{}
\mor{19}{24}{}
\mor{24}{20}{}
\mor{26}{22}{}
\mor{22}{27}{}
\mor{27}{23}{}
\mor{23}{28}{}
\mor{28}{24}{}
\mor{24}{29}{}
\enddc
$
}
\caption{Cohomology diagram for the component $\mathfrak{C}$ as in Figure \ref{fig5}.}\label{fig6}
\end{figure}

The following result follows from \cite[Thm. 6.5 , Thm. 6.6, Cor. 6.7 \& Cor. 6.10]{webb}. 

\begin{theorem}\label{webb}
Let $P^\bullet$ be an indecomposable complex in  $\mathcal{K}^b(\textup{proj-}\A)$ and let  $\mathfrak{C}$ be its corresponding component in $\Gamma(\mathcal{K}^b(\textup{proj-}\A))$. 
\begin{enumerate}
\item If $P^\bullet$ is not a stalk complex corresponding to an indecomposable projective right $\A$-module,  then the cohomology diagram of $\mathfrak{C}$ looks like as in Figure \ref{fig7}, where $A_0$ (resp. $B_0$) is the non-zero cohomology group of $P^\bullet$ of lowest (resp. higher) degree.   
\item If $P^\bullet$ is a stalk complex corresponding to the projective cover $P_S$ of a simple right $\A$-module $S$, with $P_S\not=S$, then the cohomology diagram of $\mathfrak{C}$ looks like as in Figure \ref{fig8}, where $\mathrm{H}(P_S)$ denotes the heart of $P_S$, i.e. $\mathrm{H}(P_S)= \mathrm{rad}\, P_S/\mathrm{soc}\, P_S$. Morever, if $\mathfrak{C}$ is as in Figure \ref{fig5}, then $C_0^\bullet = P_S$, and for all $n\geq 1$, $C_n^\bullet$  is the complex
\begin{equation}\label{stringcomplexsimple} 
0\to P^{-n}\xrightarrow{\delta_P^{-n}}\cdots \to P^{-2}\xrightarrow{\delta_P^{-2}}P^{-1}\xrightarrow{\delta_P^{-1}}P^0\to 0,
\end{equation}
where for all $0\leq j\leq n$, $P^{-j}=P_S$ and $\delta_P^{-j}= f_{P_S}$, where $f_{P_S}$ is as in Remark \ref{rem2.2}. 
\item If $P^\bullet$ lies in the rim of $\mathfrak{C}$ and has length $t$, then the complexes in $\mathfrak{C}$ at a distance $r$ from the rim have length $t+r$.  
\item If $P^\bullet$ has exactly two non-zero terms and $P^\bullet$ is not  (up to shifting) as in (\ref{stringcomplexsimple}) for $n=1$, then $P^\bullet$ lies in the rim of $\mathfrak{C}$. 
\end{enumerate} 
\end{theorem}

\begin{figure}
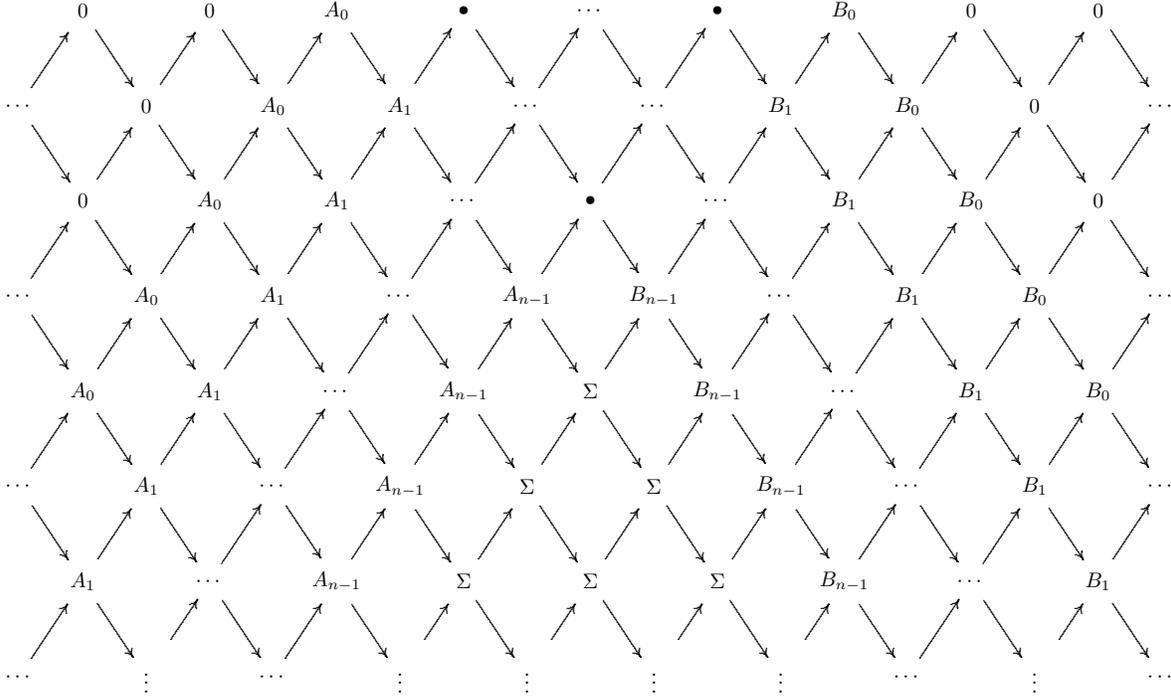

\scalebox{0.8}{
$
\begindc{\commdiag}[150]
% Second row
\obj(-16,7)[11]{$0$}
\obj(-12,7)[12]{$0$}
\obj(-8,7)[13]{$A_0$}
\obj(-4,7)[14]{$\bullet$}
\obj(0,7)[15]{$\cdots$}
\obj(4,7)[16]{$\bullet$}
\obj(8,7)[17]{$B_0$}
\obj(12,7)[18]{$0$}
\obj(16,7)[19]{$0$}
%End Second Row

%\mor{7c}{3}{}
\obj(-18,4)[20]{$\cdots$}
\obj(-14,4)[21]{$0$}
\obj(-10,4)[22]{$A_0$}
\obj(-6,4)[23]{$A_1$}
\obj(-2,4)[24]{$\cdots$}
\obj(2,4)[25]{$\cdots$}
\obj(6,4)[26]{$B_1$}
\obj(10,4)[27]{$B_0$}
\obj(14,4)[28]{$0$}
\obj(18,4)[29]{$\cdots$}

\obj(-16,1)[31]{$0$}
\obj(-12,1)[32]{$A_0$}
\obj(-8,1)[33]{$A_1$}
\obj(-4,1)[34]{$\cdots$}
\obj(0,1)[35]{$\bullet$}
\obj(4,1)[36]{$\cdots$}
\obj(8,1)[37]{$B_1$}
\obj(12,1)[38]{$B_0$}
\obj(16,1)[39]{$0$}

\obj(-18,-2)[40]{$\cdots$}
\obj(-14,-2)[41]{$A_0$}
\obj(-10,-2)[42]{$A_1$}
\obj(-6,-2)[43]{$\cdots$}
\obj(-2,-2)[44]{$A_{n-1}$}
\obj(2,-2)[45]{$B_{n-1}$}
\obj(6,-2)[46]{$\cdots$}
\obj(10,-2)[47]{$B_1$}
\obj(14,-2)[48]{$B_0$}
\obj(18,-2)[49]{$\cdots$}

\obj(-16,-5)[51]{$A_0$}
\obj(-12,-5)[52]{$A_1$}
\obj(-8,-5)[53]{$\cdots$}
\obj(-4,-5)[54]{$A_{n-1}$}
\obj(0,-5)[55]{$\Sigma$}
\obj(4,-5)[56]{$B_{n-1}$}
\obj(8,-5)[57]{$\cdots$}
\obj(12,-5)[58]{$B_1$}
\obj(16,-5)[59]{$B_0$}

\obj(-18,-8)[60]{$\cdots$}
\obj(-14,-8)[61]{$A_1$}
\obj(-10,-8)[62]{$\cdots$}
\obj(-6,-8)[63]{$A_{n-1}$}
\obj(-2,-8)[64]{$\Sigma$}
\obj(2,-8)[65]{$\Sigma$}
\obj(6,-8)[66]{$B_{n-1}$}
\obj(10,-8)[67]{$\cdots$}
\obj(14,-8)[68]{$B_1$}
\obj(18,-8)[69]{$\cdots$}

\obj(-16,-11)[71]{$A_1$}
\obj(-12,-11)[72]{$\cdots$}
\obj(-8,-11)[73]{$A_{n-1}$}
\obj(-4,-11)[74]{$\Sigma$}
\obj(0,-11)[75]{$\Sigma$}
\obj(4,-11)[76]{$\Sigma$}
\obj(8,-11)[77]{$B_{n-1}$}
\obj(12,-11)[78]{$\cdots$}
\obj(16,-11)[79]{$B_1$}

\obj(-18,-14)[80]{$\cdots$}
\obj(-14,-14)[81]{$\vdots$}
\obj(-10,-14)[82]{$\cdots$}
\obj(-6,-14)[83]{$\vdots$}
\obj(-2,-14)[84]{$\vdots$}
\obj(2,-14)[85]{$\vdots$}
\obj(6,-14)[86]{$\vdots$}
\obj(10,-14)[87]{$\cdots$}
\obj(14,-14)[88]{$\vdots$}
\obj(18,-14)[89]{$\cdots$}

\mor{11}{21}{}
\mor{12}{22}{}
\mor{13}{23}{}
\mor{14}{24}{}
\mor{15}{25}{}
\mor{16}{26}{}
\mor{17}{27}{}
\mor{18}{28}{}
\mor{19}{29}{}

\mor{20}{11}{}
\mor{21}{12}{}
\mor{22}{13}{}
\mor{23}{14}{}
\mor{24}{15}{}
\mor{25}{16}{}
\mor{26}{17}{}
\mor{27}{18}{}
\mor{28}{19}{}

\mor{20}{31}{}
\mor{21}{32}{}
\mor{22}{33}{}
\mor{23}{34}{}
\mor{24}{35}{}
\mor{25}{36}{}
\mor{26}{37}{}
\mor{27}{38}{}
\mor{28}{39}{}

\mor{31}{21}{}
\mor{32}{22}{}
\mor{33}{23}{}
\mor{34}{24}{}
\mor{35}{25}{}
\mor{36}{26}{}
\mor{37}{27}{}
\mor{38}{28}{}
\mor{39}{29}{}

\mor{31}{41}{}
\mor{32}{42}{}
\mor{33}{43}{}
\mor{34}{44}{}
\mor{35}{45}{}
\mor{36}{46}{}
\mor{37}{47}{}
\mor{38}{48}{}
\mor{39}{49}{}

\mor{40}{31}{}
\mor{41}{32}{}
\mor{42}{33}{}
\mor{43}{34}{}
\mor{44}{35}{}
\mor{45}{36}{}
\mor{46}{37}{}
\mor{47}{38}{}
\mor{48}{39}{}

\mor{40}{51}{}
\mor{41}{52}{}
\mor{42}{53}{}
\mor{43}{54}{}
\mor{44}{55}{}
\mor{45}{56}{}
\mor{46}{57}{}
\mor{47}{58}{}
\mor{48}{59}{}

\mor{51}{41}{}
\mor{52}{42}{}
\mor{53}{43}{}
\mor{54}{44}{}
\mor{55}{45}{}
\mor{56}{46}{}
\mor{57}{47}{}
\mor{58}{48}{}
\mor{59}{49}{}

\mor{51}{61}{}
\mor{52}{62}{}
\mor{53}{63}{}
\mor{54}{64}{}
\mor{55}{65}{}
\mor{56}{66}{}
\mor{57}{67}{}
\mor{58}{68}{}
\mor{59}{69}{}

\mor{60}{51}{}
\mor{61}{52}{}
\mor{62}{53}{}
\mor{63}{54}{}
\mor{64}{55}{}
\mor{65}{56}{}
\mor{66}{57}{}
\mor{67}{58}{}
\mor{68}{59}{}

\mor{60}{71}{}
\mor{61}{72}{}
\mor{62}{73}{}
\mor{63}{74}{}
\mor{64}{75}{}
\mor{65}{76}{}
\mor{66}{77}{}
\mor{67}{78}{}
\mor{68}{79}{}

\mor{71}{61}{}
\mor{72}{62}{}
\mor{73}{63}{}
\mor{74}{64}{}
\mor{75}{65}{}
\mor{76}{66}{}
\mor{77}{67}{}
\mor{78}{68}{}
\mor{79}{69}{}

\mor{71}{81}{}
\mor{72}{82}{}
\mor{73}{83}{}
\mor{74}{84}{}
\mor{75}{85}{}
\mor{76}{86}{}
\mor{77}{87}{}
\mor{78}{88}{}
\mor{79}{89}{}

\mor{80}{71}{}
\mor{81}{72}{}
\mor{82}{73}{}
\mor{83}{74}{}
\mor{84}{75}{}
\mor{85}{76}{}
\mor{86}{77}{}
\mor{87}{78}{}
\mor{88}{79}{}

\enddc
$
}
\caption{Cohomology diagram for the component $\mathfrak{C}$ as in Theorem \ref{webb} (i)}\label{fig7}
\end{figure}

\begin{figure}
\scalebox{0.75}{
$
\begindc{\commdiag}[120]
% Second row
\obj(-20,7)[2]{$\cdots$}
\obj(-16,7)[3]{$0$}
\obj(-8,7)[4]{$0$}
\obj(0,7)[5]{$P_S$}
\obj(8,7)[6]{$0$}
\obj(16,7)[7a]{$0$}
\obj(20,7)[7b]{$\cdots$}
%End Second Row

%\mor{7c}{3}{}
%Third Row
\obj(-20,4)[7c]{$\cdots$}
\obj(-12,4)[8]{$0$}
\obj(-4,4)[9]{$\mathrm{rad}\, P_S$}
\obj(4,4)[10]{$P_S/\mathrm{soc}\, P_S$}
\obj(12,4)[11]{$0$}
\obj(20,4)[11a]{$\cdots$}
%End Third Row
\mor{11}{7a}{}
\mor{7a}{11a}{}
\mor{3}{8}{}
\mor{7c}{3}{}
% Fourth Row
\obj(-20,1)[12a]{$\cdots$}
\obj(-16,1)[12]{$0$}
\obj(-8,1)[13]{$\mathrm{rad}\, P_S$}
\obj(0,1)[14]{$\mathrm{H}(P_S)$}
\obj(8,1)[15]{$P_S/\mathrm{soc}\, P_S$}
\obj(16,1)[16]{$0$}
\obj(20,1)[16a]{$\cdots$}

%End Fourth Row
\mor{7c}{12}{}
\mor{12}{8}{}
\mor{11}{16}{}
\mor{16}{11a}{}
%Fifth Row
\obj(-20,-2)[17a]{$\cdots$}
\obj(-12,-2)[17]{$\mathrm{rad}\, P_S$}
\obj(-4,-2)[18]{$\mathrm{H}(P_S)$}
\obj(4,-2)[19]{$\mathrm{H}(P_S)$}
\obj(12,-2)[20]{$P_S/\mathrm{soc}\, P_S$}
\obj(20,-2)[20a]{$\cdots$}
\mor{17a}{12}{}
\mor{12}{17}{}
\mor{20}{16}{}
\mor{16}{20a}{}

%End Fifth Row

%Sixth Row
\obj(-20,-5)[21a]{$\cdots$}
\obj(-16,-5)[21]{$\mathrm{rad}\, P_S$}
\obj(-8,-5)[22]{$\mathrm{H}(P_S)$}
\obj(0,-5)[23]{$\mathrm{H}(P_S)$}
\obj(8,-5)[24]{$\mathrm{H}(P_S)$}
\obj(16,-5)[25]{$P_S/\mathrm{soc}\, P_S$}
\obj(20,-5)[25a]{$\cdots$}
%End Sixth Row
\mor{17a}{21}{}
\mor{21}{17}{}
\mor{20}{25}{}
\mor{25}{20a}{}
%Seventh Row
\obj(-20,-8)[26a]{$\vdots$}
\obj(-12,-8)[26]{$\vdots$}
\obj(-4,-8)[27]{$\vdots$}
\obj(4,-8)[28]{$\vdots$}
\obj(12,-8)[29]{$\vdots$}
\obj(20,-8)[29a]{$\vdots$}
\mor{26a}{21}{}
\mor{21}{26}{}
\mor{29}{25}{}
\mor{25}{29a}{}

%End Seventh Row

%Morphisms

\mor{8}{4}{}
\mor{4}{9}{}
\mor{9}{5}{}
\mor{5}{10}{}
\mor{10}{6}{}
\mor{6}{11}{}
\mor{8}{13}{}
\mor{13}{9}{}
\mor{9}{14}{}
\mor{14}{10}{}
\mor{10}{15}{}
\mor{15}{11}{}
\mor{17}{13}{}
\mor{13}{18}{}
\mor{18}{14}{}
\mor{14}{19}{}
\mor{19}{15}{}
\mor{15}{20}{}
\mor{17}{22}{}
\mor{22}{18}{}
\mor{18}{23}{}
\mor{23}{19}{}
\mor{19}{24}{}
\mor{24}{20}{}
\mor{26}{22}{}
\mor{22}{27}{}
\mor{27}{23}{}
\mor{23}{28}{}
\mor{28}{24}{}
\mor{24}{29}{}
\enddc
$
}
\caption{Cohomology diagram for the component $\mathfrak{C}$ as in Theorem \ref{webb} (ii).}\label{fig8}
\end{figure}

\section{String complexes}\label{sec3}

Let $\A=\k Q/I$ be a symmetric special biserial algebra satisfying condition (C), and let $\widetilde{I}$ the admissible ideal of $\k Q$ defined as follows:
\begin{equation}
\widetilde{I}= \langle \alpha\beta, p_1, p_2\,|\, \alpha\beta, p_1-p_2\in \rho\rangle.
\end{equation}
It follows that  $\widetilde{\A}= \k Q/ \widetilde{I}$ is the associated string algebra of $\A$ as in Remark \ref{stringalg}.

In the following, we define generalized string for  $\A$. This method is inspired by that in \cite[\S 5.1]{bekmerk2}, where Bekkert et al. define generalized strings and bands for a skew-gentle algebra by using a gentle algebra associated to it.   
 
If $w$ is a path of positive length in $\A$, we define a formal inverse $w^{-1}$ of $w$ and we let $\mathbf{s}(w^{-1})=
\mathbf{t}(w)$ and $\mathbf{t}(w^{-1})=\mathbf{s}(w^{-1})$.  By a 
{\it generalized word} for $\A$ of positive length $n>0$, we mean a sequence $w_1\cdot w_2\cdots w_n$ where each $w_j$ is either a path of positive length, or the formal inverse of a path of positive length, and  such that $\mathbf{s}(w_{j+1})=\mathbf{t}(w_j)$ for $1\leq j \leq n-1$, $\mathbf{s}(w)=\mathbf{s}(w_1)$ and $\mathbf{t}(w)=\mathbf{t}(w_n)$. 
If $w=w_1\cdot w_2\cdots w_n$ is a generalized word of length $n>0$, we let $w^{-1}=w_n^{-1}\cdots w_2^{-1}\cdot w_1^{-1}$.  The concatenation of two generalized words $w=w_1\cdots w_n$ and $v=v_1\cdots v_m$ is defined as the generalized word $w\cdot v= w_1\cdots w_n\cdot 
v_1\cdots v_m$, provided that $\mathbf{s}(v)=\mathbf{t}(w)$. For all $v\in Q_0$, we consider $\1_v$ as a generalized word of length zero, and let $\1_v=\1_v^{-1}$.  

\begin{remark}
For all generalized words $w$, we assume that  $\1_{\mathbf{s}(w)}\cdot w \not= w$  and $w\cdot \1_{\mathbf{t}(w)} \not= w$. 
\end{remark}

If $w$ is a closed generalized word of non-negative length, then for all integers $n\geq 1$, we denote by $w^{\cdot n}$ the $n$-fold generalized concatenation $w\cdot w\cdots w$ of $w$ with itself. For all $v\in Q_0$, we also consider the $n$-fold generalized concatenation of $\1_v$, namely $w=\1_v^{\cdot n}$.
If $v$ and $w$ are two generalized words, we say that $w\sim_S v$ if and only if $w = v^{-1}$.

Let $J(\A)$ be the difference of ideals of $\k Q$:  
\begin{equation}\label{J(A)}
J(\A)= \widetilde{I}- \langle p_1, p_1\in \mathbf{Pa}_{>1}(\A)\,|\, p_1-p_2\in \rho\rangle. 
\end{equation}

We denote by $St(\A)$ the set of all strings for $\A$ in the sense of \cite{buri}. We denote by $\overline{GSt(\A)}$ the set of all generalized words for $\A$ of positive length $w=w_1\cdot w_2\cdots w_n$ that satisfies the following conditions. For all $1\leq j \leq n-1$, 
\begin{enumerate}
\item if $w_j, w_{j+1}\in \mathbf{Pa}_{>0}(\A)$, then $w_jw_{j+1}\in J(\A)$;
\item if $w_j^{-1}, w_{j+1}^{-1}\in \mathbf{Pa}_{>0}(\A)$, then $w_{j+1}^{-1}w_j^{-1}\in J(\A)$;
\item if either $w_j, w_{j+1}^{-1} \in \mathbf{Pa}_{>0}(\A)$ or $w_j^{-1}, w_{j+1}\in \mathbf{Pa}_{>0}(\A)$, then $w_jw_{j+1}\in St(\A)$. 
\end{enumerate}

We denote by $GSt(\A)$ a fixed set of representatives of the quotient of $\overline{GSt(\A)}$ over the equivalence relation $\sim_S$ plus all generalized words of length zero, and 
the elements of $GSt(\A)$ will be called {\it generalized strings} for $\A$.

We define inductively a function $\eta$ over the set of generalized strings for $\A$ as follows. If $w=w_1\cdot w_2\cdots w_n$ is a generalized word of positive length for $\A$ with $n\geq 1$, then for all $1\leq j\leq n$, we let
\begin{equation*}
\eta_w(j)=
\begin{cases}
0, &\text{ if $j=0$, }\\
\eta_w(j-1)+1, &\text{ if $w_j\in \mathbf{Pa}_{>0}(\A)$,}\\
\eta_w(j-1)-1, &\text { if $w_j^{-1}\in \mathbf{Pa}_{>0}(\A)$}. 
\end{cases}
\end{equation*}

\begin{definition}\label{def4.1}
Let $w=w_1\cdots w_n$ be a generalized string for $\A$ with $n\geq 1$. We define the complex $P[w]^\bullet$ in $\mathcal{K}^b(\textup{proj-}\A)$ as follows. For all $\ell\in \mathbb{Z}$, we let 

\begin{equation*}
P[w]^\ell=\bigoplus_{j=0}^n\Delta(\eta_w(j), \ell)\mathbf{P}_{c_w(j)}, 
\end{equation*}
\noindent
where $\Delta$ is the Kronecker delta, $c_w(0)=\mathbf{s}(w)$, and for all $1\leq  j \leq n$, $c_w(j)=\mathbf{t}(w_j)$.  The differential maps are $\delta_{P[w]}^i = (\delta_{jk,w}^l)_{0\leq j,k\leq n}
$, where for each $\ell\in\mathbb{Z}$, 
\begin{equation*}
\delta_{jk,w}^\ell=
\begin{cases}
p(w_{j+1}), &\text{ if $w_{j+1}\in \mathbf{Pa}_{>0}(\A)$, $\eta_w(j)=\ell$ and $k=j+1$,}\\
p(w_j^{-1}), &\text{ if $w_j^{-1}\in \mathbf{Pa}_{>0}(\A)$, $\eta_w(j)=\ell$ and $k=j-1$,}\\
0, &\text{ otherwise,}
\end{cases}
\end{equation*}
where for all $1\leq j\leq n$, and all paths $x\in \mathbf{Pa}_{>0}(\A)$, $p(w_j)$ is as in Remark \ref{projmorph}. 
We call $P[w]^\bullet$ the {\it string complex} corresponding to the generalized string $w$. 
\end{definition}

\begin{remark}\label{remext}
Let $\A$ be a symmetric special biserial algebra that satisfies the condition (C). Note that in principle, we are defining string complexes for $\widetilde{\A}$. In particular, if $\widetilde{P}^\bullet$ is a string complex over $\widetilde{\A}$, then $\widetilde{P}^\bullet$ is minimal in the sense of Remark \ref{remmin}. Thus, we can extend $\widetilde{P}^\bullet$ to be a complex $P^\bullet$ whose terms are finitely generated projective $\A$-modules by adding the missing socle to each of the biserial indecomposable direct summands of the terms of $\widetilde{P}^\bullet$ and by extending the definition of the differentials.  
\end{remark}

\begin{definition}\label{defi3.5}
\begin{enumerate}
\item For all generalized strings $w=w_1\cdot w_2\cdots w_n$ for $\A$ of positive length, we define 
\begin{equation*}
\deg w:=\max\{\eta_w(j)|0\leq j \leq n\}.
\end{equation*}
\item For all $v\in Q_0$ and all integers $n\geq 0$, if $w=\1_v^{\cdot n}$, then we let $\deg w = n$. 
\end{enumerate} 
\end{definition}

The following proposition follows by using the fact that the Nakayama's functor of a symmetric $\k$-algebra is the identity, by using the radical series of the indecomposable projective modules over symmetric special biserial algebras, and by using  Definition \ref{def4.1}.

\begin{proposition}
Let $\A$ be a symmetric special biserial algebra that satisfies condition \textup{(C)}.
\begin{enumerate}  
\item If $w$ is a string representative for $\A$ and $M[w]$ is the corresponding string right $\A$-module (in the sense of \cite{buri}), then $\mathrm{H}^0(T^{\deg w}(P[w]^\bullet))$ is isomorphic to $\Omega^{-1}M[w]$ as right $\A$-modules.
\item For all generalized string representatives $w$ for $\A$, the string complexes $P[w]^\bullet$ and $T^{-\deg w} (P[w^{-1}]^\bullet)$ are isomorphic in $\mathcal{K}^b(\textup{proj-}\A)$. 
\end{enumerate}
\end{proposition}

\begin{remark}\label{topsocle}
If $\A$ is a symmetric special biserial algebra (not necessarily satisfying the condition (C)), then we denote by $P[\1_v^{\cdot n}]^\bullet$ the complex as in (\ref{stringcomplexsimple}). In particular, $P[\1_v^{\cdot n}]^\bullet$ is indecomposable in $\mathcal{K}^b(\textup{proj-}\A)$, and the component $\mathfrak{C}$ of the Auslander-Reiten quiver of  $\mathcal{K}^b(\textup{proj-}\A)$ containing $\mathbf{P}_v$ can be described completely by using the complexes $P[\1_v^{\cdot n}]^\bullet$ together with Figure \ref{fig5}.  Observe in particular that in this situation, $\mathbf{P}_i\not= P[\1_i]^\bullet$ as complexes.
\end{remark}

Let  $\A$ be $\A_4$ as in Figure \ref{fig1}.  It was proved in \cite[Thm. 8]{giraldo-velez}) that if $P^\bullet$ is a string complex over $\A$, then $P^\bullet$ is an indecomposable object in $\mathcal{K}^b(\textup{proj-}\A)$. This was obtained by defining a functor of $\k$-linear categories $\mathbf{F}_\A:\mathcal{K}^b(\textup{proj-}\A)\to\mathscr{S}(\mathscr{Y}(\A), \k)$, where $\mathscr{S}(\mathscr{Y}(\A), \k)$ is the $\k$-category of Bondarenko's representations of a linearly ordered set $\mathscr{Y}(\A)$ determined by $\A$ (see e.g. \cite[\S 2 \& \S4.2]{bekmerk}), such that $\mathbf{F}_\A$ identifies $P^\bullet$ with an indecomposable representation in $\mathscr{S}(\mathscr{Y}(\A), \k)$ (see \cite[Thm. 3]{bekmerk}). Similar arguments were used by A. Franco et al. in \cite{franco-giraldo-rizzo} in order to describe combinatorially classes of indecomposable objects in the derived category over string algebras satisfying certain non-trivial conditions. In particular, the following result follows from  \cite[Thm. 27]{franco-giraldo-rizzo}.

\begin{theorem}\label{thm3.7}
Let $\widetilde{\A} = \k Q/\widetilde{I}$ be a string algebra with the property that every arrow belongs to a unique maximal path. Then every string complex over $\widetilde{\A}$ is indecomposable in $\mathcal{K}^b(\textup{proj-}\widetilde{\A})$.
\end{theorem}

%The following result follows from \cite[Thm. 27]{franco-giraldo-rizzo} and Definition \ref{classC}.

%\begin{theorem}
%Let $\A$ be a symmetric special biserial algebra that belongs to the class $\mathscr{C}$ as in Definition \ref{classC}. If $P^\bullet$ is a string complex for $\A$ as in Definition \ref{def4.1}, then $P^\bullet$ is indecomposable in $\mathcal{K}^b(\textup{proj-}\A)$.
%\end{theorem}

\section{Auslander-Reiten components containing a string complex}\label{sec4}

Assume that $\A$ is a symmetric special biserial algebra that belongs to the class $\mathscr{C}$ as in Definition \ref{classC}. 

For all complexes $P^\bullet$ in $\mathcal{K}^b(\textup{proj-}\A)$, we denote by $^-P^\bullet$ the complex obtained from $P^\bullet$ by changing the signs of the differentials.

\begin{definition}\label{defPk}
Let  $w$ be a generalized string representative for $\A$, and let $P[w]^\bullet$ be the corresponding string complex. For all $k\geq 0$, we define a perfect complex $P_k[w]^\bullet$ as follows. 
We let $P_{-1}[w]^\bullet=0$, $P_0[w]^\bullet =P[w]^\bullet$ and $P_1[w]^\bullet=\mathrm{cone}(f_{w,0}^\bullet)$, where $f_{w,0}^\bullet:{^-P_0[w]^\bullet}\to P_0[w]^\bullet$ is the morphism in $\mathcal{K}^b(\textup{proj-}\A)$ with  $f^{\deg w}_{w,0}=f_{P_0[w]^{\deg w}}$ and $f^\ell_{w,0}=0$ for all $\ell\not= \deg w$, where $f_{P_0[w]^{\deg w}}$ and $\deg w$ are as in Remark \ref{rem2.2} and Definition \ref{defi3.5}, respectively.  
Let $k\geq 2$ be fixed and assume that $P_{k-1}[w]^\bullet$ is previously defined. Consider  $f^\bullet_{w,k-1}$ the morphism $f^\bullet_{w,k-1}: {^-}P_{k-1}[w]^\bullet\to P_{k-1}[w]^\bullet$, where $f^{\deg w}_{w,k-1}=f_{P_{k-1}[w]^{\deg w}}$, and $f^\ell_{w,k-1}=0$ for all $\ell\not=\deg w$. Then $P_k[w]^\bullet$ is the indecomposable complex in $\mathcal{K}^b(\textup{proj-}\A)$ such that 
\begin{equation}\label{P_k}
\mathrm{cone}(f^\bullet_{w,k-1})=P_k[w]^\bullet\oplus T(P_{k-2}[w]^\bullet).
\end{equation}
\end{definition}

In the following, we state and prove the main result of this article. Its proof uses similar ideas to those in the proof of \cite[Thm. 14]{giraldo-velez}.

\begin{theorem}\label{prop4.2}
Let $\A=\k Q/I$ be a symmetric special biserial algebra in the class $\mathscr{C}$ as in Definition \ref{classC}, and let $w$ be a generalized string representative for $\A$ of positive length. Then we have the following.

\begin{enumerate}
\item The string complex $P[w]^\bullet$ is indecomposable in $\mathcal{K}^b(\textup{proj-}\A)$.
\item If $\mathfrak{C}$ is the component of  $\Gamma(\mathcal{K}^b(\textup{proj-}\A))$ containing $P[w]^\bullet$ as in Figure \ref{fig5}, then for all $k\geq 0$,  $C_k^\bullet=P_k[w]^\bullet$, where $P_k[w]^\bullet$ is as in Definition \ref{defPk}. 
\end{enumerate}
\end{theorem}

\begin{proof}
Since $\A$ belongs to the class $\mathscr{C}$ as in Definition \ref{classC}, the statement (i) follows from Theorem \ref{thm3.7} and Remark \ref{remext}. 

(ii). We first prove by induction on the length of $w$ that $P[w]^\bullet$ lies in the rim of $\mathfrak{C}$. If $w$ has length $1$, then $P[w]$ has two non-zero terms. It follows by Theorem \ref{webb} (iv) that $P[w]^\bullet$ lies on the rim of $\mathfrak{C}$. Let $n\geq 2$ be arbitrary and assume by induction that for all generalized string representatives $w'$ that have length less than $n$, the perfect complex $P[w']^\bullet$ lies on the rim of $\mathfrak{C}$.  Assume next that $w$ has length $n$ and suppose by contradiction that $P[w]^\bullet$ does not lie on the rim of $\mathfrak{C}$. Then for some $k_0\geq 1$ and some $m\in \Z$, we obtain that $T^m(P[w]^\bullet)=C_{k_0}^\bullet$. Because of the general structure of the component $\mathfrak{C}$ as shown in Figure \ref{fig5}, we get that there exists an irreducible morphism $u_{k_0-1}^\bullet:C_{k_0-1}^\bullet \to T^m(P[w]^\bullet)$. By \cite[Thm. 6]{giraldo-merklen} (see also \cite[Lemma 3.1(1)]{scherotzke}), we can assume that $u_{k_0-1}^\bullet$ is an irreducible morphism in the additive category of bounded complexes $\mathcal{C}^b(\textup{proj-}\A)$. It follows from Theorem \ref{webb} (iii) and \cite[Prop. 3]{giraldo-merklen} that for all $j\in \Z$, $u_{k_0-1}^j$ is a monomorphism and thus $T^{-m}(C_{k_0-1}^\bullet)$ is isomorphic in $\mathcal{C}^b(\textup{proj-}\A)$ to one of the proper subcomplexes of $P[w]^\bullet$, which are themselves either of the form $T^{-m'}(\mathbf{P}_v)$ for some $v\in Q_0$ and $m'\in \Z$, or a string complex $P[w']^\bullet$, where $w'$ is a proper generalized substring of $w$ for $\A$ of positive length. Assume first that $T^{-m}(C_{k_0-1}^\bullet)=T^{-m'}(\mathbf{P}_v)$ for some $v\in Q_0$ and $m'\in \Z$. By Theorem \ref{webb} (ii), $k_0=1$ and $P[w]^\bullet = T^{-m''}(P[\1_v^{\cdot 2}]^\bullet)$ for some $m''\in \Z$, where $P[\1_v^{\cdot 2}]^\bullet$ is as in Remark \ref{topsocle}. This argument contradicts that $w$ is a generalized string for $\A$ of positive length.  Next assume that $T^{-m}(C_{k_0-1}^\bullet)=P[w']^\bullet$, where $w'$ is a proper generalized substring of $w$ of positive length. By induction we have that $P[w']^\bullet$ lies in the rim of $\mathfrak{C}$. %Note that since the length of $P[w]^\bullet$ is $n+1$, it follows from \cite[Cor. 6.7]{webb} that the length of $P[w']^\bullet$ is $n$. 
Thus by Theorem \ref{webb},  the non-zero cohomology groups of higher (resp. lowest) degree of $P[w]^\bullet$ and $P[w']^\bullet$ are isomorphic. This is impossible for $w'$ is assumed to be a proper generalized substring of $w$ of positive length.  Hence $P[w]^\bullet$ lies in the boundary of $\mathfrak{C}$. 

Note that after shifting, we can assume that $C_0^\bullet= P[w]^\bullet=P_0[w]^\bullet$.

Let $k\geq 1$ be fixed but arbitrary, and assume that for all $0\leq j\leq k-1$, we have $C_j^\bullet = P_j[w]^\bullet$. Once again, from the general structure of the component $\mathfrak{C}$ as shown in Figure \ref{fig5}, we get an Auslander-Reiten triangle
\begin{equation*}
 T^{-1}(C_{k-1}^\bullet)\to T^{-1}(\mathrm{cone}(h_{k-1}^\bullet))\to C_{k-1}^\bullet\xrightarrow{h_{k-1}^\bullet} C_{k-1}^\bullet, 
\end{equation*}
for some morphism $h_{k-1}^\bullet:C_{k-1}^\bullet\to C_{k-1}^\bullet$ in $\mathcal{K}^b(\textup{proj-}\A)$. By Lemma \ref{lem4.1}, we can assume that $h_{k-1}^\bullet = f_{w,k-1}^\bullet$. Since $\mathrm{cone}(f_{w,k-1}^\bullet)=P_k[w]^\bullet\oplus T( P_{k-2}[w]^\bullet)=P_k[w]^\bullet\oplus T(C_{k-2}^\bullet)$, we obtain that $C_k^\bullet=P_k[w]^\bullet$. This finishes the proof of Theorem \ref{prop4.2}.

\end{proof}

\begin{example}\label{exam3}

Let $\A$ be the symmetric special biserial $\k$-algebra $\A_3$ as in Figure \ref{fig1}. As noted before, $\A$ belongs to the class $\mathscr{C}$ as in Definition \ref{classC}. Let $w$ be the generalized string $w = \zeta_0^{-1}\cdot \tau_0\cdot \tau_1$. Then the corresponding string complex is of the following form:

\begin{align*}
P[w]^\bullet: 0\to \mathbf{P}_0\to \mathbf{P}_0\oplus \mathbf{P}_1\to \mathbf{P}_2\to 0, && \text{(in degrees $-1$, $0$ and $1$).}
\end{align*}
 By using the notation in Definition \ref{defPk}, we have that $P_0[w]^\bullet = P[w]^\bullet$, and for $k=1,2,3$, the complex $P_k[w]^\bullet$ is of the following form:

\begin{align*}
P_1[w]^\bullet&: 0\to \mathbf{P}_0\to \mathbf{P}_0\oplus \mathbf{P}_1\oplus \mathbf{P}_0\to \mathbf{P}_2\oplus\mathbf{P}_0\oplus \mathbf{P}_1\to \mathbf{P}_2\to 0,\\
&\text{(in degrees $-2$, $-1$, $0$ and $1$);}\\
P_2[w]^\bullet&: 0\to \mathbf{P}_0\to \mathbf{P}_0\oplus \mathbf{P}_1\oplus \mathbf{P}_0\to \mathbf{P}_2\oplus\mathbf{P}_0\oplus \mathbf{P}_1\oplus\mathbf{P}_0\to \mathbf{P}_2\oplus\mathbf{P}_0\oplus \mathbf{P}_1\to \mathbf{P}_2\to 0,\\
&\text{(in degrees $-3$, $-2$, $-1$, $0$ and $1$);}\\
P_3[w]^\bullet&: 0\to \mathbf{P}_0\to \mathbf{P}_0\oplus \mathbf{P}_1\oplus \mathbf{P}_0\to \mathbf{P}_2\oplus\mathbf{P}_0\oplus \mathbf{P}_1\oplus\mathbf{P}_0\to \mathbf{P}_2\oplus\mathbf{P}_0\oplus\mathbf{P}_1\oplus \mathbf{P}_0\to \mathbf{P}_2\oplus \mathbf{P}_0\oplus \mathbf{P}_1\to \mathbf{P}_2\to 0,\\
&\text{(in degrees $-4$, $-3$, $-2$, $-1$, $0$ and $1$).}\\
\end{align*}

\begin{figure}
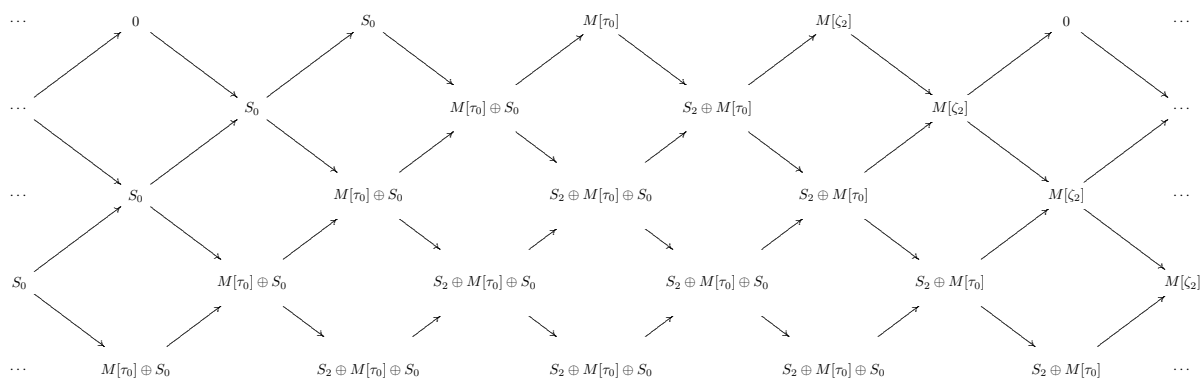

\scalebox{0.55}{
$
\begindc{0}[200]
% Second row
\obj(-20,7)[2]{$\cdots$}
\obj(-16,7)[3]{$0$}
\obj(-8,7)[4]{$S_0$}
\obj(0,7)[5]{$M[\tau_0]$}
\obj(8,7)[6]{$M[\zeta_2]$}
\obj(16,7)[7a]{$0$}
\obj(20,7)[7b]{$\cdots$}
%End Second Row

%\mor{7c}{3}{}
%Third Row
\obj(-20,4)[7c]{$\cdots$}
\obj(-12,4)[8]{$S_0$}
\obj(-4,4)[9]{$M[\tau_0]\oplus S_0$}
\obj(4,4)[10]{$S_2\oplus M[\tau_0]$}
\obj(12,4)[11]{$M[\zeta_2]$}
\obj(20,4)[11a]{$\cdots$}
%End Third Row
\mor{11}{7a}{}
\mor{7a}{11a}{}
\mor{3}{8}{}
\mor{7c}{3}{}
% Fourth Row
\obj(-20,1)[12a]{$\cdots$}
\obj(-16,1)[12]{$S_0$}
\obj(-8,1)[13]{$M[\tau_0]\oplus S_0$}
\obj(0,1)[14]{$S_2\oplus M[\tau_0]\oplus S_0$}
\obj(8,1)[15]{$S_2\oplus M[\tau_0]$}
\obj(16,1)[16]{$M[\zeta_2]$}
\obj(20,1)[16a]{$\cdots$}

%End Fourth Row
\mor{7c}{12}{}
\mor{12}{8}{}
\mor{11}{16}{}
\mor{16}{11a}{}
%Fifth Row
\obj(-20,-2)[17a]{$S_0$}
\obj(-12,-2)[17]{$M[\tau_0]\oplus S_0$}
\obj(-4,-2)[18]{$S_2\oplus M[\tau_0]\oplus S_0$}
\obj(4,-2)[19]{$S_2\oplus M[\tau_0]\oplus S_0$}
\obj(12,-2)[20]{$S_2\oplus M[\tau_0]$}
\obj(20,-2)[20a]{$M[\zeta_2]$}
\mor{17a}{12}{}
\mor{12}{17}{}
\mor{20}{16}{}
\mor{16}{20a}{}

%End Fifth Row

%Sixth Row
\obj(-20,-5)[21a]{$\cdots$}
\obj(-16,-5)[21]{$M[\tau_0]\oplus S_0$}
\obj(-8,-5)[22]{$S_2\oplus M[\tau_0]\oplus S_0$}
\obj(0,-5)[23]{$S_2\oplus M[\tau_0]\oplus S_0$}
\obj(8,-5)[24]{$S_2\oplus M[\tau_0]\oplus S_0$}
\obj(16,-5)[25]{$S_2\oplus M[\tau_0]$}
\obj(20,-5)[25a]{$\cdots$}
%End Sixth Row
\mor{17a}{21}{}
\mor{21}{17}{}
\mor{20}{25}{}
\mor{25}{20a}{}
%Seventh Row

%End Seventh Row

%Morphisms

\mor{8}{4}{}
\mor{4}{9}{}
\mor{9}{5}{}
\mor{5}{10}{}
\mor{10}{6}{}
\mor{6}{11}{}
\mor{8}{13}{}
\mor{13}{9}{}
\mor{9}{14}{}
\mor{14}{10}{}
\mor{10}{15}{}
\mor{15}{11}{}
\mor{17}{13}{}
\mor{13}{18}{}
\mor{18}{14}{}
\mor{14}{19}{}
\mor{19}{15}{}
\mor{15}{20}{}
\mor{17}{22}{}
\mor{22}{18}{}
\mor{18}{23}{}
\mor{23}{19}{}
\mor{19}{24}{}
\mor{24}{20}{}
\enddc
$
}
\caption{Cohomology diagram for the component $\mathfrak{C}$ containing $P[w]^\bullet$ as in Example \ref{exam3}. }\label{fig10}
\end{figure}
%\end{landscape}

The cohomology diagram (as in Theorem \ref{webb} (i)) corresponding to the component $\mathfrak{C}$ containing $P[w]^\bullet$ is shown in Figure \ref{fig10}, where  $M[\tau_0]$ and $M[\zeta_2]$ are the string $\A$-modules corresponding to the strings $\tau_0$ and $\zeta_2$ for $\A$ in the sense of \cite{buri}. Note that the cohomology group $\Sigma$, which Webb calls the {\it stabilizing cohomology group}, is given by $S_2\oplus M[\tau_0]\oplus S_0$. Moreover, from this description it follows that for $k\geq 4$, $P_k[w]^\bullet$ has the following terms: 
\begin{equation*}
P_k[w]^\ell=\begin{cases}
0, &\text{ if $\ell<-(k+1)$ or $\ell>1$,}\\
\mathbf{P}_0, &\text{if $\ell=-(k+1)$,}\\
\mathbf{P}_0\oplus \mathbf{P}_1\oplus \mathbf{P}_0, &\text{if $\ell = -k$,}\\
\mathbf{P}_2\oplus\mathbf{P}_0\oplus \mathbf{P}_1\oplus\mathbf{P}_0, &\text{if $-k +1 \leq \ell \leq -1$,}\\
\mathbf{P}_2\oplus \mathbf{P}_0\oplus \mathbf{P}_1, &\text{if $\ell=0$,}\\
\mathbf{P}_2, &\text{ if $\ell=1$.}
\end{cases}
\end{equation*}
\end{example}

\bibliographystyle{amsplain}
\bibliography{StringComplexesGiraldoRuedaVelez}

\end{document}